\documentclass[10pt,reqno]{amsart}
\usepackage{bbm}
\usepackage{tabu}
\usepackage{amsmath}
\usepackage{mathrsfs}
\usepackage{bm}
\usepackage{amsfonts} %%% i.e. use 12pt type
\usepackage{graphicx,latexsym,bm,amsmath,amssymb,verbatim,multicol,lscape}
\usepackage{enumerate}
\newtheorem{thm}{Theorem} [section]

\newtheorem{cor}[thm]{Corollary}
\newtheorem{lem}[thm]{Lemma}
\newtheorem{prop}[thm]{Proposition}

\theoremstyle{definition}
\newtheorem{defn}[thm]{Definition}
\theoremstyle{remark}
\newtheorem{rem}[thm]{Remark}
\newtheorem{exa}[thm]{Example}

\numberwithin{equation}{section}
\newcommand{\fq}{{\mathbb F}_{q}}
\newcommand{\fp}{{\mathbb F}_{p}}

\newcommand{\fth}{{\mathbb F}_{3}}

\newcommand{\Tr}{\mbox{Tr}}

\newcommand{\rmv}[1]{}

\def\<{\left\langle}
\def\>{\right\rangle}

\begin{document}

\title[On binomial Weil sums and an application]
{On binomial Weil sums and an application}%
\author{Kaimin Cheng}
\address{School of Mathematics and Information, China West Normal University,
Nanchong, 637002, P. R. China}
\email{ckm20@126.com}
\author{Shuhong Gao}
\address{School of Mathematical and Statistic Sciences, Clemson University, Clemson, SC 29631, USA}
\email{sgao@clemson.edu}
\thanks{This work was supported partially by the NSF of China (No. 12226335).}

%\subjclass{Primary 11T22,11R18}%
\date{\today}%zxzsasszxsmkht5rrfc  vasv
%\dedicatory{}%
%\commby{}%
% ----------------------------------------------------------------
\begin{abstract}
Let $p$ be a prime, and $N$ be a positive integer not divisible by $p$. Denote by ${\rm ord}_N(p)$ the multiplicative order of $p$ modulo $N$. Let $\mathbb{F}_q$ represent the finite field of order $q=p^{{\rm ord}_N(p)}$. For $a, b\in\mathbb{F}_q$, we define a binomial exponential sum by
$$S_N(a,b):=\sum_{x\in\mathbb{F}_q\setminus\{0\}}\chi(ax^{\frac{q-1}{N}}+bx),$$
where $\chi$ is the canonical additive character of $\mathbb{F}_q$. In this paper, we provide an explicit evaluation of  $S_{N}(a,b)$ for any odd prime $p$ and any $N$ satisfying ${\rm ord}_{N}(p)=\phi(N)$. Our elementary and direct approach allows for the construction of a class of ternary linear codes, with their exact weight distribution determined. Furthermore, we prove that the dual codes achieve optimality with respect to the sphere packing bound, thereby generalizing previous results from even to odd characteristic fields.\end{abstract}

\maketitle

\section{Introduction}
Let $p$ be a prime number and $q$ be a power of $p$. Let $\mathbb{F}_q$ be the finite field with $q$ elements. We denote by ${\rm Tr}$ the absolute \textit{trace} function from $\mathbb{F}_q$ onto $\mathbb{F}_p$. A \textit{Weil sum} is an exponential sum of the form
$$\sum_{x\in \mathbb{F}_q}\chi(f(x)),$$
where $\chi$ is a non-trivial additive character of $\mathbb{F}_q$ and $f(x)\in\mathbb{F}_q[x]$. Explicit evaluation of Weil sums with specific forms has been investigated extensively (see, for example, \cite{[Car1],[Car2],[Cou],[FH],[Moi],[Moi2],[Wan1],[Wan2]}), but it is still quite difficult in general. Results that estimate the absolute value of the sum are more common and have appeared regularly for many years. A comprehensive survey covering many aspects of this topic including theories and applications was given in the books \cite{[LN]} and \cite{[BGMV]}.

The first goal of this paper is to obtain the explicit values of a binomial Weil sum. The canonical additive character of $\mathbb{F}_q$, denoted by $\chi_1$, is defined by
$$\chi_1(x)=e^{\frac{2\pi i{\rm Tr}(x)}{p}}$$
for all $x\in\mathbb{F}_q$. Due to the properties of the trace function, $\chi_1(x+y)=\chi_1(x)\chi_1(y)$ and $\chi_1(x^p)=\chi_1(x)$ for all $x,y\in\mathbb{F}_q$, any additive character $\chi_{\gamma}$ of $\mathbb{F}_q$ can be obtained from $\chi_1$: for any $\gamma\in\mathbb{F}_q$, $\chi_{\gamma}(x)=\chi_1(\gamma x)$ for all $x\in\mathbb{F}_q$. So one only needs to explicitly evaluate the Weil sums with $\chi=\chi_1$ as the Weil sums for any non-trivial additive character can be derived simply by manipulating the results obtained using this identity. Let $N$ be a positive integer not divisible by $p$. Let $d$ be the multiplicative order of $p$ modulo $N$, denoted by $d={\rm ord}_N(p)$; that is, $d$ is the least positive integer such that $p^d\equiv 1\pmod{N}$. Let $q=p^d$. For $a, b\in\mathbb{F}_q$, define a Weil sum of a particular shape by
$$S_N(a,b):=\sum_{x\in\mathbb{F}_q^*}\chi(ax^{\frac{q-1}{N}}+bx)$$
with $\chi=\chi_1$. Providing an explicit evaluation of $S_N(a,b)$ is of significant interest. In 2009, Moisio \cite{[Moi2]} gave the explicit evaluation of $S_{N}(a,b)$ for $p=2$ and $N=\wp^m$, where $\wp$ is an odd prime and $m$ is a positive integer such that $2$ is a primitive root modulo $\wp^m$. In 2019, under certain conditions, Wu, Yue and Li \cite{[WYL]} generalized the evaluation result of Moisio to odd characteristics. In fact, they presented an evaluation of $S_{\wp^m}(a,b)$ for $p\ge 3$ under a condition that the least subfield $\mathbb{F}_{p^t}$ of $\mathbb{F}_{q}$ subject to certain restrictions is $\mathbb{F}_{p}$ or $\mathbb{F}_{p^2}$, where $\wp$ is an odd prime and $m$ is a positive integer such that $p$ is a primitive root modulo $\wp^m$.  In this paper, inspired by the work \cite{[Moi2]}, we obtain an explicit evaluation of $S_N(a,b)$ for any positive integer $N$ and any odd prime $p$, assuming that  $p$ is a primitive root modulo $N$. More precisely, we shall compute $S_N(a,b)$ for $N=2, 4, \wp^m$ and $2\wp^m$ respectively, which are the all possible values of $N$ since there exists a primitive root modulo $N$ if and only if $N=2, 4, \wp^m$ or $2\wp^m$, where $\wp$ is an odd prime and $m$ is a positive integer. Our results are not based on any conditions, and our method is elementary and straightforward, differing significantly from that used in \cite{[WYL]}.

The second objective of the paper is to construct a class of linear codes. Let $n,k$ and $d$ be positive integers. An $[n,k,d]$ $p$-ary linear code is a $k$-dimensional subspace of $\mathbb{F}_p$ with minimum (Hamming) distance $d$. For a positive integer $i$, denote by $A_i$ the number of codewords with Hamming weight $i$ in a code $\mathcal{C}$ of length $n$.
Let
$$1+A_1z+A_2z^2+\cdots+A_nz^n$$
be the \textit{weight enumerate} of $\mathcal{C}$.
The sequence $(A_1,A_2,\ldots,A_n)$ is referred to as the\textit{ weight distribution} of
$\mathcal{C}$. If the number of nonzero $A_i$ in the sequence $(A_1,A_2,\ldots,A_n)$ is $N$, then the code $\mathcal{C}$ is called an $N$-weight code. The weight distribution directly determines the minimum distance of the code and the error correcting capability. It also contains key information on computation of the probability of error detection and correction with respect to some error detection and correction algorithms \cite{[Klo]}. As a result, studying the weight distribution of a linear code is an important research topic in coding theory. For a set $D=\{d_1,d_2,\ldots,d_n\}\subseteq\mathbb{F}_q^*$, define a set by
\begin{align}\label{c1-1}
\mathcal{C}_D:=\{({\rm Tr}(d_1x),{\rm Tr}(d_2x),\ldots,{\rm Tr}(d_nx)):\ x\in\mathbb{F}_q\}.
\end{align}
Clearly, $\mathcal{C}_D$ is a linear codes of length over $\mathbb{F}_p$, and we call $D$ the \textit{defining set} of the code $\mathcal{C}_D$. This defining-set construction is fundamental since every linear code over $\mathbb{F}_p$ can be expressed as $\mathcal{C}_D$ for some defining set $D$ (possibly multiset) \cite{[Din]}. Many classes of codes with few weights were produced by properly selecting the defining set $D\subseteq\mathbb{F}_q$ (see, for example, \cite{[Din1],[Din2],[TXF],[WDX]}). In the paper, a class of two-weight codes is constructed motivated by the work of Wang, Ding and Xue \cite{[WDX]}, and the weight distribution is determined. The dual codes is also studied, and it is proved to be optimal codes. The linear codes with two weights presented in this paper have applications in secret sharing \cite{[ADHK]}, authentication codes \cite{[DW]}, combinatorial designs \cite{[CK]}, etc.

Throughout this paper, we always assume that $p$ is an odd prime, $N$ is a positive integer not divisible by $p$ such that ${\rm ord}_{N}(p)=\phi(N)$, and $q=p^{\phi(N)}$. The paper is organized as follows. First in Section 2, we evaluate $S_N(a,0)$. Subsequently in Section 3, the explicit evaluation of the exponential sum $S_N(a,b)$ with $ab\ne 0$ is obtained. In Section 4, a class of ternary codes is constructed, and its weight distribution is determined. Finally in Section 5, we give the parameters of the dual codes of the codes obtained in Section 4, and show that the dual codes are optimal with respect to the sphere packing bound.

\section{Evaluation of $S_N(a,0)$}
It is well known that a primitive root modulo $N$ exists if and only if $N=2, 4, \wp^m$ or $2\wp^m$, where $\wp$ is an odd prime and $m$ is a positive integer. Let $p$ be an odd prime and $N$ be a positive integer not divisible by $p$ such that $p$ is a primitive root modulo $N$, i.e., ${\rm ord}_{N}(p)=\phi(N)$. Then $N$ must be $2, 4, \wp^m$ or $2\wp^m$ for some odd prime $\wp$ and positive integer $m$. Let $q=p^{\phi(N)}$ and $g$ be a primitive element of $\fq$. First, if $N=2$, then $q=p$. It implies that
\begin{align*}
S_2(a,b)=\sum_{x\in\fp^*}\chi(ax^{\frac{p-1}{2}}+bx)
=\sum_{i=0}^{p-2}\chi(a(g^{\frac{p-1}{2}})^{i}+bg^i)
=\sum_{i=0}^{p-2}\zeta_p^{a(-1)^i+bg^i},
\end{align*}
which is computable once $p$ is given, where $\zeta_p$ is a primitive $p$-th root of unity.
\begin{exa}\label{exa2.1}
Let $p=q=5$. We know that $2$ is a primitive element of $\mathbb{F}_5$. Then
$$S_2(a,b)
=\sum_{i=0}^{p-2}\zeta_5^{a(-1)^i+b\times2^i}=\zeta_5^{a+b}+\zeta_5^{-a+2b}+
\zeta_5^{a-b}+\zeta_5^{-a-2b}$$
for any $a,b\in\mathbb{F}_5$.
\end{exa}
Next, we provide the explicit evaluation of $S_N(a,0)$ for $N=4, \wp^m$ and $2\wp^m$ respectively. Let $\xi:=g^{\frac{q-1}{N}}$, a primitive $N$-th root of unity. For $a\in\mathbb{F}_q$, define a sum
\begin{align}\label{c2-0}
S_N(a)=\sum_{i=0}^{N-1}\chi(a\xi^i).
\end{align}
One then readily finds that
\begin{align}\label{c2-2-2}
S_N(a,0)=\frac{q-1}{N}S_N(a).
\end{align}
In fact, by definition, we have
\begin{align*}
S_N(a,0)=\sum_{x\in\mathbb{F}_q^*}\chi(ax^{\frac{q-1}{N}})
=\sum_{i=0}^{q-2}\chi(a(g^i)^{\frac{q-1}{N}})
=\sum_{i=0}^{q-2}\chi(a\xi^{i})
=\frac{q-1}{N}\sum_{i=0}^{N-1}\chi(a\xi^i)
=\frac{q-1}{N}S_N(a).
\end{align*}
Here, the second to last equality holds since $\xi$ is a primitive $N$-th root of unity. So in order to evaluate $S_N(a,0)$, we only need to compute $S_N(a)$.

Let $Q_{N}(x)$ be the $N$-th cyclotomic polynomial over $\mathbb{F}_p$. Clearly, $\xi$ is a root of $Q_{N}(x)$. Note that $p$ is a primitive root modulo $N$. It then follows from a well-known result (see \cite[Theorem 2.47(ii)]{[LN]}) that $Q_{N}(x)$ factors into exactly one irreducible polynomial over $\mathbb{F}_p$. This implies that $Q_{N}(x)$ is irreducible over $\mathbb{F}_p$, and $\xi$ is of degree $\phi(N)$ over $\mathbb{F}_p$. Note that $q=p^{\phi(N)}$. It implies that
$$\mathbb{F}_q=\mathbb{F}_p(\xi).$$
Hence, $\{\xi,\xi^2,\ldots,\xi^{\phi(N)}\}$ can be a basis of $\mathbb{F}_q$ over $\mathbb{F}_p$. Then, for each $a\in\mathbb{F}_q$ there exists a unique vector $(a_1,a_2,\ldots,a_{\phi(N)})\in\mathbb{F}_p^{\phi(N)}$ such that
$$a=\sum_{i=1}^{\phi(N)}a_i\xi^i.$$
So we can identify $a$ with the vector $(a_1,a_2,\ldots,a_{\phi(N)})$.

The main results of this section are given as follows.
\begin{thm}\label{thm2.1.1}
Assume that $p$ is the primitive root modulo $4$ and $q=p^2$. Let $\xi=g^{\frac{q-1}{4}}$. For any $a\in\fq$, if write $a=a_1\xi+a_2\xi^2$ with $a_1,a_2\in\fp$, then
$$S_4(a)=\zeta_p^{2a_1}+\zeta_p^{-2a_1}+\zeta_p^{2a_2}+\zeta_p^{-2a_2}.$$
\end{thm}
\begin{proof}
Note that $x^2+1$ is the minimal polynomial of $\xi$ over $\fp$. It follows that ${\rm Tr}(\xi)=0$. So, by the linearity of trace functions we have that
$$S_4(a)=\sum_{i=0}^3\chi(a\xi^i)=\sum_{i=0}^3\zeta_p^{{\rm Tr}(a_1\xi^{i+1}+a_2\xi^{i+2})}
=\zeta_p^{2a_1}+\zeta_p^{-2a_1}+\zeta_p^{2a_2}+\zeta_p^{-2a_2},$$
as desired.
\end{proof}
\begin{thm}\label{thm2.1}
Assume that $p$ is the primitive root modulo $\wp^m$ with $\wp$ an odd prime and $m$ a positive integer, and $q=p^{\phi(\wp^m)}$. Let $\xi=g^{\frac{q-1}{\wp^m}}$. For any $a\in\mathbb{F}_q$, if write
$a=a_1\xi+a_2\xi^2+\cdots+a_{\phi(\wp^m)}\xi^{\phi(\wp^m)}$
with each $a_s\in\mathbb{F}_p$ and let
\begin{align}\label{c2-1}
e_i(a)=\prod_{k=1}^{\wp-1}\zeta_p^{-\wp^{m-1}a_{k\wp^{m-1}-i}},\ f_i(a)=\sum_{k=1}^{\wp-1}\zeta_p^{\wp^{m}a_{k\wp^{m-1}-i}}
\end{align}
for any integer $i$ with $0\le i\le \wp^{m-1}-1$,
then
$$S_{\wp^m}(a)=\sum_{i=0}^{\wp^{m-1}-1}e_i(a)(f_i(a)+1).$$
\end{thm}
\begin{thm}\label{thm2.1.1}
Assume that $p$ is the primitive root modulo $2\wp^m$ with $\wp$ an odd prime and $m$ a positive integer, and $q=p^{\phi(2\wp^m)}$. Let $\xi=g^{\frac{q-1}{2\wp^m}}$. For any $a\in\mathbb{F}_q$, if write
$a=a_1\xi+a_2\xi^2+\cdots+a_{\phi(\wp^m)}\xi^{\phi(\wp^m)}$
with each $a_s\in\mathbb{F}_p$,
then we have
$$S_{2\wp^m}(a)=\sum_{i=0}^{\wp^{m-1}-1}\left(\zeta_p^{\Delta
-i}+\zeta_p^{-\Delta_i}+\sum_{t=1}^{\wp-1}\left(\delta_{i,t}+\delta_{i,t}^{-1}\right)\right),$$
where
\begin{align*}
\Delta_i=-\wp^{m-1}\sum_{k=1}^{\wp-1}(-1)^ka_{k\wp^{m-1}-i}\ \text{and}\ \delta_{i,t}=\zeta_p^{(-1)^t\Delta_i+\wp^ma_{t\wp^{m-1}-i}}.
\end{align*}
\end{thm}
Before proving Theorems \ref{thm2.1} and \ref{thm2.1.1}, we present some useful lemmas as follows.
\begin{lem}\label{lem2.2}
\cite[Theorem 3.6]{[Nat]} Let $u$ be an odd prime, and let $a\ne \pm 1$ be an integer not divisible by $u$. Let $d$ be the order of $a$ modulo $u$. Let $k_0$ be the largest integer such that $a^d\equiv 1\pmod{u^{k_0}}$. For any positive integer $k$, we then have
$${\rm ord}_{u^k}(a)=\begin{cases}
d,&\text{if}\ 1\le k\le k_0,\\
du^{k-k_0},&\text{if}\ k>k_0.
\end{cases}$$
\end{lem}
\begin{lem}\label{lem2.3}
Assume that $p$ is the primitive root modulo $\wp^m$ with $\wp$ an odd prime and $m$ a positive integer, and $q=p^{\phi(\wp^m)}$. Let $\xi=g^{\frac{q-1}{\wp^m}}$. Then for any nonnegative integer $j$ we have
$${\rm Tr}(\xi^j)=\begin{cases}
(\wp-1)\wp^{m-1},& \text{if}\ \wp^m\mid j,\\
-\wp^{m-1},& \text{if}\ \wp^{m-1}\mid j\ \text{but}\ \wp^m\nmid j,\\
0,& \text{otherwise}.
\end{cases}$$
\end{lem}
\begin{proof}
It is easy to see that Lemma \ref{lem2.2} holds for $j=0$. In the following, let $j\ge 1$ and write $j=\wp^s j'$ with $s\ge 0$ and $\gcd(j',\wp)=1$. The remaining proof can be divided into two cases as follows.

{\sc case 1}. $s\ge m$. One then has $\xi^j=1$ since $\xi$ is a primitive $\wp^m$-th root of unity. So ${\rm Tr}(\xi^j)={\rm Tr}(1)=\phi(\wp^m)$.

{\sc case 2}. $s<m$. First, we know that $\xi^j$ is a primitive $\wp^{m-s}$-th root of unity. Moreover, one claims that $p$ is a primitive root modulo $\wp^{m-s}$ for all $s=0,1,\ldots,m-1$ since $p$ is a primitive root modulo $\wp^{m}$. Obviously, the claim is true for $m=1$. Now let $m\ge 2$. Let ${\rm ord}_{\wp}(p)=d$, and let $k_0$ be the largest integer such that $\wp^{k_0}\mid (p^d-1)$. Note that  $\gcd(d,\wp)=1$ since $d\mid \phi(\wp)$, and ${\rm ord}_{\wp^{m}}(p)=(\wp-1)\wp^{m-1}$. It then follows from Lemma \ref{lem2.2} that $k_0<m$ and $(\wp-1)\wp^{m-1}=d\wp^{m-k_0}$. This forces that $k_0=1$ and $d=\wp-1$. By using Lemma \ref{lem2.2} again, we have that ${\rm ord}_{\wp^{m-s}}(p)=(\wp-1)\wp^{m-s-1}=\phi(\wp^{m-s})$ for each $s$ with $0\le s\le m-1$. Hence, the claim is true. From this claim, together with the result of \cite[Theorem 2.47(ii)]{[LN]}, we have that the $\wp^{m-s}$-th cyclotomic polynomial $Q_{\wp^{m-s}}(x)$ is irreducible over $\mathbb{F}_p$ for any $0\le s\le m-1$. Note that $Q_{\wp^{m-s}}(\xi^{j})=0$. It implies that $Q_{\wp^{m-s}}(x)$ is the minimal polynomial of $\xi^{j}$ over
$\mathbb{F}_p$. It is checked that
$$Q_{\wp^{m-s}}(x)=x^{(\wp-1)\wp^{m-s-1}}+x^{(\wp-2)\wp^{m-s-1}}
+\cdots+x^{\wp^{m-s-1}}+1.$$
Therefore, for any positive integer $n_1$ and $n_2$ with $n_1\mid n_2$ and $n_2\mid \phi(\wp^m)$, if let ${\rm Tr}_{n_1,n_2}$ be the trace function from $\mathbb{F}_{p^{n_2}}$ onto $\mathbb{F}_{p^{n_1}}$, then
$${\rm Tr}_{1,\phi(\wp^{m-s})}(\xi^{j})=\begin{cases}
-1,&\text{if}\ s=m-1,\\
0,&\text{otherwise}.
\end{cases}$$
It follows from the transitivity of traces that
$${\rm Tr}(\xi^{j})={\rm Tr}_{1,\phi(\wp^m)}(\xi^{j})
={\rm Tr}_{1,\phi(\wp^{m-s})}\left({\rm Tr}_{\phi(\wp^{m-s}),\phi(\wp^{m})}(\xi^{j})\right)
=\frac{\phi(\wp^m)}{\phi(\wp^{m-s})}{\rm Tr}_{1,\phi(\wp^{m-s})}(\xi^{j}),$$
which is $-\wp^{m-1}$ if $s=m-1$, or $0$ otherwise.

Putting {\sc cases} 1 and 2 together, the result of Lemma \ref{lem2.3} is obtained.
\end{proof}
\begin{lem}\label{lem2.5.1}
Let $\wp_1$ and $\wp_2$ be two odd positive integers with $\gcd(\wp_1,\wp_2)=1$. Let $d_1$ and $d_2$ be positive integers. Then $\wp_1^{d_1}\equiv 1\pmod{\wp_2^{d_2}}$ if and only if $\wp_1^{d_1}\equiv 1\pmod{2\wp_2^{d_2}}$. Consequently, $${\rm ord}_{\wp_2^{m}}(\wp_1)={\rm ord}_{2\wp_2^{m}}(\wp_1)$$ for any positive  integer $m$.
\end{lem}
\begin{proof}
Since $\wp_1$ and $\wp_2$ are odd, $\wp_2^{d_2}\mid (\wp_1^{d_1}-1)$ if and only if $2\wp_2^{d_2}\mid (\wp_1^{d_1}-1)$. Lemma \ref{lem2.5.1} follows immediately.
\end{proof}
\begin{lem}\label{lem2.6.1}
Assume that $p$ is the primitive root modulo $2\wp^m$ with $\wp$ an odd prime and $m$ a positive integer, and $q=p^{\phi(2\wp^m)}$. Let $\xi=g^{\frac{q-1}{2\wp^m}}$. Then for any nonnegative integer $j$ we have
$${\rm Tr}(\xi^j)=\begin{cases}
(\wp-1)\wp^{m-1},& \text{if}\ 2\wp^m\mid j,\\
-(\wp-1)\wp^{m-1},& \text{if}\ \wp^m\mid j\ \text{but}\ 2\nmid j,\\
-\wp^{m-1},& \text{if}\ 2\wp^{m-1}\mid j\ \text{but}\ \wp^m\nmid j,\\
\wp^{m-1},& \text{if}\ \wp^{m-1}\mid j\ \text{but}\ \wp^m\nmid j,\ 2\nmid j,\\
0,& \text{otherwise}.
\end{cases}$$
\end{lem}
\begin{proof}
The proof of Lemma \ref{lem2.6.1} can be carried out by the main approach used in proving Lemma \ref{lem2.3}. The following are the details. First of all, if $j=0$, then ${\rm Tr}(\xi^j)={\rm Tr}(1)=\phi(2\wp^m)=(\wp-1)\wp^{m-1}$. Now let $j\ge 1$ in what follows, and write $j=2^e\wp^sj'$ with $e,s$ nonnegative integers and $j$ an odd positive integer coprime to $\wp$.\\
$\bullet$ Let $e\ge 1$ and $s\ge m$. In this case $\xi^j=1$. So ${\rm Tr}(\xi^j)=(\wp-1)\wp^{m-1}$.\\
$\bullet$ Let $e=0$ and $s\ge m$. One computes that $\xi^j=(g^{\frac{q-1}{2\wp^m}})^j=(g^{\frac{q-1}{2}})^{\wp^{s-m}j'}=-1$. So ${\rm Tr}(\xi^j)=-(\wp-1)\wp^{m-1}$.\\
$\bullet$ Let $e\ge 1$ and $s<m$. We know that $\xi^j=(g^{\frac{q-1}{2\wp^m}})^j=(g^{\frac{q-1}{\wp^{m-s}}})^{2^{e-1}j'}$. So ${\rm Tr}(\xi^j)$ is a primitive $\wp^{m-s}$-th root of unity. On the other hand, note that $p$ is a primitive root modulo $2\wp^m$. By Lemma \ref{lem2.5.1}, $p$ is a primitive root modulo $\wp^m$. So from the second case of the proof of Lemma \ref{lem2.3}, we derive that
$${\rm Tr}(\xi^j)=\begin{cases}
-\wp^{m-1},&\text{if}\ s=m-1,\\
0,&\text{if}\ s<m-1.
\end{cases}$$\\
$\bullet$ Let $e=0$ and $s<m$. One has that $\xi^j=(g^{\frac{q-1}{2\wp^m}})^j=(g^{\frac{q-1}{2\wp^{m-s}}})^{j'}$. So ${\rm Tr}(\xi^j)$ is a primitive $2\wp^{m-s}$-th root of unity. By Lemma \ref{lem2.5.1}, $p$ is a primitive root modulo $\wp^m$ since $p$ is a primitive root modulo $2\wp^m$. It follows from the claim in the proof of Lemma \ref{lem2.3} that $p$ is a primitive root modulo $\wp^{m-s}$. From Lemma \ref{lem2.5.1} we know that $p$ is also a primitive root modulo $2\wp^{m-s}$. Therefore, the $2\wp^{m-s}$-th cyclotomic polynomial $Q_{2\wp^{m-s}}(x)$ is irreducible over $\fp$. Consequently, $Q_{2\wp^{m-s}}(x)$ is the minimal polynomial of $\xi^j$ over $\fp$. Note that $Q_{2\wp^{m-s}}(x)=Q_{\wp^{m-s}}(-x)$ (see \cite[Exercise 2.57(d)]{[LN]}). So one arrives at
$${\rm Tr}(\xi^j)=\begin{cases}
\wp^{m-1},&\text{if}\ s=m-1\\
0,&\text{if}\ s<m-1
\end{cases}$$
by following the similar way as the proof of Lemma \ref{lem2.3}.

Combining all above, we complete the proof of Lemma \ref{lem2.6.1}.
\end{proof}

\noindent{\textit{Proof of Theorem \ref{thm2.1}}.} Let $\zeta_p$ be the primitive $p$-th root of unity. From the division algorithm, we know that for each integer $j$ with $0\le j\le \wp^m-1$ there exists a unique integer pair $(t,i)$ with $0\le t\le \wp-1, 0\le i\le \wp^{m-1}-1$ such that $j=t\wp^{m-1}+i$. It implies that
\begin{align}\label{c2-2}
S_{\wp^m}(a)=\sum_{j=0}^{\wp^m-1}\chi(a\xi^j)
=\sum_{t=0}^{\wp-1}\sum_{i=0}^{\wp^{m-1}-1}\chi(a\xi^{t\wp^{m-1}+i})
=\sum_{t=0}^{\wp-1}\sum_{i=0}^{\wp^{m-1}-1}\zeta_p^{{\Tr}(a\xi^{t\wp^{m-1}+i})}.
\end{align}
Note that
\begin{align}\label{c2-3}
{\rm Tr}(a\xi^{t\wp^{m-1}+i})=
\sum_{s=1}^{\phi(\wp^m)}a_s{\rm Tr}(\xi^{t\wp^{m-1}+i+s})
\end{align}
since $a=\sum_{s=1}^{\phi(\wp^m)}a_s\xi^s$. Let $W$ be a set defined by
$$W:=\left\{(i,t,s)\in \mathbb{N}^3:\ i\le\wp^{m-1}-1,t\le \wp-1\ \text{and}\ 1\le s\le \phi(\wp^m)\right\}.$$
For each $(i,t,s)\in W$, let us compute ${\rm Tr}(\xi^{t\wp^{m-1}+i+s})$ in the following.
Define two subsets of $W$ as
$$W_1:=\left\{(i,t,s)\in W:\ \wp^{m-1}\mid (t\wp^{m-1}+i+s)\right\}$$
and
$$W_2:=\left\{(i,t,s)\in W:\ \wp^m\mid (t\wp^{m-1}+i+s)\right\}.$$
On the one hand, it is clear that
\begin{align}\label{c2-4}
W_1=\left\{(i,t,s)\in W:\ s=k\wp^{m-1}-i\ \text{with}\ 1\le k\le \wp-1\right\}.
\end{align}
On the other hand, for each $(i,t,s)\in W$, since $1\le t\wp^{m-1}+i+s<2\wp^m$ one has that $\wp^m\mid (t\wp^{m-1}+i+s)$ if and only if $t\wp^{m-1}+i+s=\wp^m$
if and only if $s=(\wp-t)\wp^{m-1}-i$. This means that
\begin{align}\label{c2-5}
W_2=\left\{(i,t,s)\in W:\ s=(\wp-t)\wp^{m-1}-i\ \text{and}\ t\ne 0\right\}.
\end{align}
It then follows from Lemma \ref{lem2.3} that for any $(i,t,s)\in W$
\begin{align}\label{c2-6}
{\rm Tr}(\xi^{t\wp^{m-1}+i+s})=\begin{cases}
(\wp-1)\wp^{m-1},& \text{if}\ (i,t,s)\in W_2,\\
-\wp^{m-1},& \text{if}\ (i,t,s)\in W_1\setminus W_2,\\
0,& \text{otherwise}.
\end{cases}
\end{align}
 Combining (\ref{c2-3})-(\ref{c2-6}), we derive that
\begin{align}\label{c2-7}
{\rm Tr}(a\xi^{t\wp^{m-1}+i})=
-\wp^{m-1}\sum_{k=1}^{\wp-1}a_{k\wp^{m-1}-i}+\epsilon,
\end{align}
where $\epsilon$ is equal to $0$ if $t=0$, and $\wp^{m}a_{(\wp-t)\wp^{m-1}-i}$ otherwise.
For convenience, define a notation $E(x):=\zeta_p^x$. So, putting (\ref{c2-7}) into (\ref{c2-2}), we have
\begin{align*}
S_{\wp^m}(a)=&\sum_{i=0}^{\wp^{m-1}-1}E\big(
-\wp^{m-1}\sum_{k=1}^{\wp-1}a_{k\wp^{m-1}-i}\big)\\
&\quad +\sum_{t=1}^{\wp-1}
\sum_{i=0}^{\wp^{m-1}-1}E\big(\wp^{m}a_{(\wp-t)\wp^{m-1}-i}-\wp^{m-1}
\sum_{k=1}^{\wp-1}a_{k\wp^{m-1}-i}\big)\\
&=\sum_{i=0}^{\wp^{m-1}-1}E\big(
-\wp^{m-1}\sum_{k=1}^{\wp-1}a_{k\wp^{m-1}-i}\big)
\big(1+\sum_{t=1}^{\wp-1}E\big(\wp^{m}
a_{(\wp-t)\wp^{m-1}-i}\big)\big)\\
&=\sum_{i=0}^{\wp^{m-1}-1}e_i(a)(f_i(a)+1),
\end{align*}
as desired.
The proof of Theorem \ref{thm2.1} is concluded.\hfill$\Box$\\

Next, we proceed to prove Theorem \ref{thm2.1.1}.

\noindent{\textit{Proof of Theorem \ref{thm2.1.1}}.} For any integer $j$ with $0\le j\le 2\wp^m-1$, let $(t,i)$ be the unique integer pair with $0\le t\le 2\wp-1, 0\le i\le \wp^{m-1}-1$ such that $j=t\wp^{m-1}+i$. It follows that
\begin{align}\label{c2-2-1}
S_{2\wp^m}(a)=\sum_{t=0}^{2\wp-1}\sum_{i=0}^{\wp^{m-1}-1}
\chi(a\xi^{t\wp^{m-1}+i})
=\sum_{t=0}^{2\wp-1}\sum_{i=0}^{\wp^{m-1}-1}\zeta_p^{{\rm Tr}(a\xi^{t\wp^{m-1}+i})}.
\end{align}
Note that
\begin{align}\label{c2-3-1}
{\rm Tr}(a\xi^{t\wp^{m-1}+i})=
\sum_{s=1}^{\phi(\wp^m)}a_s{\rm Tr}(\xi^{t\wp^{m-1}+i+s})
\end{align}
since $a=\sum_{s=1}^{\phi(\wp^m)}a_s\xi^s$. In the following, one would like to compute ${\rm Tr}(\xi^{t\wp^{m-1}+i+s})$. For this purpose, let $V:=\{(i,t,s)\in\mathbb{N}:\ i\le \wp^{m-1}-1,t\le 2\wp-1,1\le s\le\phi(\wp^m)\}$, and define four subsets of $V$ by
$$V_1:=\{(i,t,s)\in V:\ 2\wp^m\mid (t\wp^{m-1}+i+s)\},$$
$$V_2:=\{(i,t,s)\in V:\ \wp^m\mid (t\wp^{m-1}+i+s)\ \text{but}\ 2\nmid (t\wp^{m-1}+i+s)\},$$
$$V_3:=\{(i,t,s)\in V:\ 2\wp^{m-1}\mid (t\wp^{m-1}+i+s)\ \text{but}\ \wp^m\nmid (t\wp^{m-1}+i+s)\}$$
and
$$V_4:=\{(i,t,s)\in V:\ \wp^{m-1}\mid (t\wp^{m-1}+i+s)\ \text{but}\ \wp^m\nmid (t\wp^{m-1}+i+s)\ \text{and}\ 2\nmid (t\wp^{m-1}+i+s)\}.$$
Note that $1\le t\wp^{m-1}+i+s\le 3\wp^m-\wp^{m-1}$. So $V_1$ and $V_2$ can be simplified to
$$V_1=\{(i,t,s)\in V:\ s=(2\wp-t)\wp^{m-1}-i,\wp+1\le t\le 2\wp-1\}$$
and
$$V_2=\{(i,t,s)\in V:\ s=(\wp-t)\wp^{m-1}-i,1\le t\le \wp-1\}.$$
For $(i,t,s)\in V$, one observes that $2\wp^{m-1}\mid (t\wp^{m-1}+i+s)$ if and only if $s=k\wp^{m-1}-i$ with $k$ satisfying $1\le k\le\wp-1$ and having the same parity as $t$. It then follows that
$$V_3=\{(i,t,s)\in V:\ s=k\wp^{m-1}-i, 1\le k\le \wp-1,\ \text{and}\ k\equiv t\pmod{2}\}\setminus(V_1\cup V_2)$$
and
$$V_4=\{(i,t,s)\in V:\ s=k\wp^{m-1}-i, 1\le k\le \wp-1,\ \text{and}\ k\not\equiv t\pmod{2}\}\setminus(V_1\cup V_2).$$
Therefore, by Lemma \ref{lem2.6.1} and Equation (\ref{c2-3-1}), we have that
\begin{align}\label{c2-11-1}
{\rm Tr}(a\xi^{t\wp^{m-1}+i})=(-1)^{t-1}\wp^{m-1}\sum_{k=1}^{\wp-1}
(-1)^ka_{k\wp^{m-1}-i}+\epsilon_t,
\end{align}
where
$$\epsilon_t=
\begin{cases}
0,&\text{if}\ t=0\ \text{or}\ \wp,\\
-\wp^ma_{(\wp-t)\wp^{m-1}-i},&\text{if}\ 1\le t\le \wp-1,\\
\wp^ma_{(2\wp-t)\wp^{m-1}-i},&\text{if}\ \wp+1\le t\le 2\wp-1.
\end{cases}$$
Substituting (\ref{c2-11-1}) into (\ref{c2-2-1}) yields the desired result. This completes the proof of Theorem \ref{thm2.1.1}.\hfill$\Box$\\

Assume that $p$ is the primitive root modulo $\wp^m$ with $\wp$ an odd prime and $m$ a positive integer, and $q=p^{\phi(\wp^m)}$. Let $\xi=g^{\frac{q-1}{\wp^m}}$ (or $\xi=g^{\frac{q-1}{2\wp^m}}$). For $a\in\mathbb{F}_{q}$, let $a=a_1\xi+a_2\xi^2+\cdots+a_{\phi(\wp^m)}\xi^{\phi(\wp^m)}$
with each $a_s\in\mathbb{F}_p$. For any integers $i,j$ with $0\le i\le \wp^{m-1}-1$ and $0\le j\le p-1$, let
\begin{align}\label{c2-2-10}
a^{(i)}:=(a_{\wp^{m-1}-i},a_{2\wp^{m-1}-i},\ldots,
a_{(\wp-1)\wp^{m-1}-i})\in\mathbb{F}_p^{\wp-1}
\end{align}
and $A_j(a^{(i)})$ be the number of components of $a^{(i)}$ equal $j$. The next two results
are special cases of Theorem \ref{thm2.1}.
\begin{cor}\label{cor2.4}
Assume that $p$ is the primitive root modulo $\wp^m$ with $\wp$ an odd prime and $m$ a positive integer, and $q=p^{\phi(\wp^m)}$. Then
$$S_{\wp^m}(a)=(\wp-1)\zeta_p^{-a\wp^{m-1}}+\zeta_p^{a(\wp-1)\wp^{m-1}}+\wp^m-\wp$$
for any $a\in\mathbb{F}_p$, where $\zeta_p$ is a primitive $p$-th root of unity.
\end{cor}
\begin{proof}
Let $a\in\mathbb{F}_p$. Let $g$ be a primitive element of $\mathbb{F}_q$. We already know that the $\wp^m$-th cyclotomic polynomial $Q_{\wp^m}$ is the minimal polynomial of $\xi=g^{\frac{q-1}{\wp^m}}$ over $\mathbb{F}_p$. Then
$$1=-\xi^{\wp^{m-1}}-\xi^{2\wp^{m-1}}-\cdots-\xi^{(\wp-1)\wp^{m-1}}.$$
It follows that
$$a=-a\xi^{\wp^{m-1}}-a\xi^{2\wp^{m-1}}-\cdots-a\xi^{(\wp-1)\wp^{m-1}}.$$
So we have that
\begin{align}\label{c2-9}
a^{(0)}=(\underbrace{-a,-a,\ldots,-a}_{\wp-1}),\ \text{and}\
a^{(i)}=(\underbrace{0,0,\ldots,0}_{\wp-1})
\end{align}
for any $1\le i\le \wp^{m-1}-1$. Substituting (\ref{c2-9}) into (\ref{c2-1}) yields that
$$e_i(a)(f_i(a)+1)=\begin{cases}
(\wp-1)\zeta_p^{-a\wp^{m-1}}+\zeta_p^{a(\wp-1)\wp^{m-1}},&\text{if}\ i=0,\\
\wp,&\text{if}\ i\ne 0.
\end{cases}$$
Therefore the desired result follows from Theorem \ref{thm2.1}.
\end{proof}
\begin{cor}\label{cor2.5}
Assume that $3$ is the primitive root modulo $\wp^m$ with $\wp$ an odd prime and $m$ a positive integer, and $q=3^{\phi(\wp^m)}$. Then for any $a\in\mathbb{F}_q$ we have
$$S_{\wp^m}(a)=\begin{cases}
\sum_{i=0}^{\wp^{m-1}-1}u_iv_i(\zeta_3),&\text{if}\ \wp\equiv 1\pmod{3},\\
\sum_{i=0}^{\wp^{m-1}-1}u_iv_i(\zeta_3^2),&\text{if}\ \wp\equiv -1\pmod{3}\ \text{and}\ 2\nmid m,\\
\sum_{i=0}^{\wp^{m-1}-1}u_i^{-1}v_i(\zeta_3),&\text{if}\ \wp\equiv -1\pmod{3}\ \text{and}\ 2\mid m,
\end{cases}$$
where $u_i:=\zeta_3^{-A_1(a^{(i)})-2A_2(a^{(i)})}$ and $v_i(x):=1+A_0(a^{(i)})+A_1(a^{(i)})x+
A_2(a^{(i)})x^{2}$.  In particular,
\begin{align}\label{c2-11}
S_{\wp^m}(1)=\begin{cases}
(\wp-1)\zeta_3^2+\wp^m-\wp+1,&\text{if}\ \wp\equiv 1\pmod{3},\\
(\wp-1)\zeta_3^2+\zeta_3+\wp^m-\wp,&\text{if}\ \wp\equiv -1\pmod{3}\ \text{and}\ 2\nmid m,\\
\zeta_3^2+(\wp-1)\zeta_3+\wp^m-\wp,&\text{if}\ \wp\equiv -1\pmod{3}\ \text{and}\ 2\mid m,
\end{cases}
\end{align}
and
\begin{align}\label{c2-12}
S_{\wp^m}(2)=\begin{cases}
(\wp-1)\zeta_3+\wp^m-\wp+1,&\text{if}\ \wp\equiv 1\pmod{3},\\
\zeta_3^2+(\wp-1)\zeta_3+\wp^m-\wp,&\text{if}\ \wp\equiv -1\pmod{3}\ \text{and}\ 2\nmid m,\\
(\wp-1)\zeta_3^2+\zeta_3+\wp^m-\wp,&\text{if}\ \wp\equiv -1\pmod{3}\ \text{and}\ 2\mid m.
\end{cases}
\end{align}
Here, $\zeta_3$ is a primitive $3$-th root of unity.
\end{cor}

\section{Evaluation of $S_n(a,b)$}
Throughout this section we let $N=4, \wp^m$ or $2\wp^m$ with $\wp$ an odd prime and $m$ a positive integer and assume that $p$ is the primitive root modulo $N$, and $q=p^{\phi(N)}$. In this section, we are going to evaluate $S_N(a,b)$ with $b\ne 0$ over $\fq$. Clearly,
$$S_N(0,b)=\begin{cases}
-1,&\text{if}\ b\ne 0,\\
q-1,&\text{if}\ b=0.
\end{cases}$$
So we only need to focus on the case $ab\ne 0$. First, we present a result due to Moisio.
\begin{lem}\label{lem3.1}
\cite[Theorem 1]{[Moi]} Let $p$ be a prime and $q=p^e$ with $e$ a positive integer. Let $g$ be a primitive element of $\mathbb{F}_q$. If $e=2sd$ and $n\mid (p^d+1)$ with $s,n$ and $d$ being positive integers, then for any $c\in\mathbb{F}_q^*$
$$\sum_{x\in\mathbb{F}_q}\chi(cx^n)=\begin{cases}
(-1)^sp^{sd},&\text{if}\ {\rm ind}_g(c)\not\equiv k\pmod{n},\\
(-1)^{s-1}(n-1)p^{sd},&\text{if}\ {\rm ind}_g(c)\equiv k\pmod{n},
\end{cases}$$
where ${\rm ind}_g(c)$ denotes the index of $c$ with respect to $g$, and $k=0$ if $p=2$; or $p>2$, $2\mid s$; or $p>2$, $2\nmid s$ and $2\mid (p^d+1)/n$, and $k=\frac{n}{2}$ if $p>2$, $2\nmid s$ and $2\nmid (p^d+1)/n$.
\end{lem}
If let $d=\frac{\phi(N)}{2}$, one then checks that $N\mid (p^d+1)$. So, in Lemma \ref{lem3.1}, by taking $s=1$, $d=\frac{\phi(N)}{2}$, we have
\begin{lem}\label{lem3.2}
Let $c\in\mathbb{F}_q^*$. Then
$$\sum_{x\in\mathbb{F}_q^*}\chi(cx^{N})=\begin{cases}
(N-1)\sqrt{q}-1,&\text{if}\ c\in H,\\
-\sqrt{q}-1,&\text{if}\ c\notin H,
\end{cases}$$
where $H$ is the subgroup of $\mathbb{F}_q^*$ generated by $g^{N}$.
\end{lem}
We now evaluate $S_n(a,b)$ for $ab\ne0$.
\begin{thm}\label{thm3.3}
Let $N=4, \wp^m$ or $2\wp^m$ with $\wp$ an odd prime and $m$ a positive integer and assume that $p$ is the primitive root modulo $N$, and $q=p^{\phi(N)}$. Then we have
$$S_N(a,b)=\chi(c)\sqrt{q}-\frac{\sqrt{q}+1}{\wp^m}S_N(c)$$
for any $a,b\in \mathbb{F}_q^*$, where $c=ab^{-\frac{q-1}{N}}$ and $S_N(c)$ is defined as (\ref{c2-0}).
\end{thm}
\begin{proof}
Let $H$ be a subgroup of the multiplicative group $\mathbb{F}_q^*$ generated by $g^{N}$. Then we know that the set of cosets of $\mathbb{F}_q^*$ modulo $H$ is
$$\{g^iH:\ 0\le i\le N-1\},$$
which is a partition of $\mathbb{F}_q^*$. It follows that
$$\mathbb{F}_q^*=\bigcup_{i=0}^{N-1}g^iH.$$
Recall that $\xi=g^{\frac{q-1}{N}}$. Hence, we deduce that
\begin{align}\label{c3-1}
S_N(a,b)&=\sum_{x\in\mathbb{F}_q^*}\chi(ax^{\frac{q-1}{N}}+bx)\nonumber\\
&=\sum_{x\in\mathbb{F}_q^*}\chi(cx^{\frac{q-1}{N}}+x)\nonumber\\
&=\sum_{i=0}^{N-1}\sum_{h\in H}\chi\left(c(g^ih)^{\frac{q-1}{N}}\right)\chi(g^ih)\nonumber\\
&=\sum_{i=0}^{N-1}\sum_{h\in H}\chi(c\xi^{i})\chi(g^ih)\nonumber\\
&=\sum_{i=0}^{N-1}\chi(c\xi^{i})\sum_{h\in H}\chi(g^ih)\nonumber\\
&=\frac{1}{N}\sum_{i=0}^{N-1}\chi(c\xi^{i})\sum_{x\in \mathbb{F}_q^*}\chi(g^ix^{N}).
\end{align}
Note that $g^i\in H$ with some $0\le i\le N-1$ if and only if $i=0$. So, by employing Lemma \ref{lem3.2} into Equation (\ref{c3-1}) we arrive at
\begin{align*}
S_N(a,b)&=\frac{1}{N}\Big(\chi(c)\sum_{x\in \mathbb{F}_q^*}\chi(x^{N})+\sum_{i=1}^{N-1}\chi(c\xi^{i})\sum_{x\in \mathbb{F}_q^*}\chi(g^ix^{N})\Big)\\
&=\frac{1}{N}\Big(\left((N-1)\sqrt{q}-1\right)\chi(c)-(\sqrt{q}+1)
\sum_{i=1}^{N-1}\chi(c\xi^{i})\Big)\\
&=\frac{1}{N}\left(\left((N-1)\sqrt{q}-1\right)\chi(c)-(\sqrt{q}+1)
\left(S_N(c)-\chi(c)\right)\right)\\
&=\chi(c)\sqrt{q}-\frac{\sqrt{q}+1}{N}S_N(c).
\end{align*}
This completes the proof of Theorem \ref{thm3.3}.
\end{proof}
Let $F(x)=ax^{\frac{q-1}{N}}+L(x)$, where $a\in\fq^*$ and $L(x)$ is a linearized polynomial of the form
$$L(x)=\sum_{i=0}^{\phi(N)-1}b_ix^{p^i}$$
with $b_i\in\fq$ for all $i$. Then a relationship between two Weil sums can be established as follows.
\begin{prop}\label{prop3.4}
Let $b=\sum_{i=0}^{\phi(n)-1}b_i^{p^{\phi(n)-i}}$. Then
$$\sum_{x\in\fq^*}\chi(F(x))=S_N(a,b)$$
for any $a\in\fq$.
\end{prop}
\begin{proof}
Let $b=\sum_{i=0}^{\phi(n)-1}b_i^{p^{N-i}}$. Then we have
\begin{align*}
\sum_{x\in\fq^*}\chi(F(x))&=\sum_{x\in\fq^*}
\chi\big(ax^{\frac{q-1}{N}}+L(x)\big)\\
&=\sum_{x\in\fq^*}
\chi\big(ax^{\frac{q-1}{N}}\big)
\chi\big(\sum_{i=0}^{\phi(N)-1}b_ix^{p^i}\big)\\
&=\sum_{x\in\fq^*}
\chi\big(ax^{\frac{q-1}{N}}\big)
\prod_{i=0}^{\phi(N)-1}\chi\big(b_ix^{p^i}\big)\\
&=\sum_{x\in\fq^*}
\chi\big(ax^{\frac{q-1}{N}}\big)
\prod_{i=0}^{\phi(N)-1}\chi\big(b_i^{p^{\phi(N)-i}}x\big)\\
&=\sum_{x\in\fq^*}
\chi\big(ax^{\frac{q-1}{N}}\big)
\chi\Big(x\sum_{i=0}^{\phi(N)-1}b_i^{p^{\phi(N)-i}}\Big)\\
&=\sum_{x\in\fq^*}
\chi\left(ax^{\frac{q-1}{N}}+bx\right)\\
&=S_N(a,b),
\end{align*}
as desired.
\end{proof}
Therefore, the Weil sum with $f(x)=F(x)$ can be evaluated eventually.

\section{A ternary code and its weight distribution}
Let $m$ be a positive integer. Let $\wp$ be an odd prime such that $3$ is a primitive root modulo $\wp^m$. Let $q=3^{\phi(\wp^m)}$, and Tr be the trace from $\mathbb{F}_q$ onto $\mathbb{F}_3$. In this section, we investigate the linear code $\mathcal{C}_D$ given in (\ref{c1-1}) with the defining set
\begin{align}\label{c4-1}
D=\{x\in\mathbb{F}_q^*: \Tr(x^{\frac{q-1}{\wp^m}})=0\}.
\end{align}
Actually, we have the following result.
\begin{thm}\label{thm4.1}
The following statements on $\mathcal{C}_D$ are true.
\begin{enumerate}[(a)]
\item If $\wp\equiv 1\pmod{3}$, then $\mathcal{C}_D$ is an $\left[\frac{(q-1)(\wp^m-\wp+1)}{\wp^m}, (\wp-1)\wp^{m-1}\right]$ ternary code having two weights $w_1=\frac{2(\wp^m-\wp+1)q}{3\wp^m}-\frac{2(\wp-1)\sqrt{q}}{3\wp^m},\ w_2=\frac{2(\wp^m-\wp+1)q}{3\wp^m}+\frac{2(\wp^m-\wp+1)\sqrt{q}}{3\wp^m}$ with enumerates $A_{w_1}=\frac{(\wp^m-\wp+1)(q-1)}{\wp^m}, \ A_{w_2}=\frac{(\wp-1)(q-1)}{\wp^m}$.
\item If $\wp\equiv -1\pmod{3}$ and $m\ge 2$, then $\mathcal{C}_D$ is an $\left[\frac{(q-1)(\wp^{m-1}-1)}{\wp^{m-1}}, (\wp-1)\wp^{m-1}\right]$ ternary code having two weights $w_1=\frac{2(\wp^{m-1}-1)q}{3\wp^{m-1}}+\frac{2(\wp^{m-1}-1)\sqrt{q}}{3\wp^{m-1}},\ w_2=\frac{2(\wp^{m-1}-1)q}{3\wp^{m-1}}-\frac{2\sqrt{q}}{3\wp^{m-1}}$ with enumerates $A_{w_1}=\frac{q-1}{\wp^{m-1}}, \ A_{w_2}=\frac{(\wp^{m-1}-1)(q-1)}{\wp^{m-1}}$.
\end{enumerate}
\end{thm}
\begin{proof}
Let $n=\#D$, the length of the code $\mathcal{C}_D$. Let
$$n_0=\#\{x\in\mathbb{F}_q: \Tr(x^{\frac{q-1}{\wp^m}})=0\}.$$
Clearly, $n=n_0-1$. First of all, we compute $n$ by studying $n_0$. Note that
$$n_0=\frac{1}{3}\sum_{x\in\fq}\sum_{y\in\fth}\zeta^{y{\rm Tr}(x^{\frac{q-1}{\wp^m}})},$$
where $\zeta$ is a primitive $3$-th root of unity. It then follows that
\begin{align}\label{c4-2}
n_0&=3^{\phi(\wp^m)-1}+\frac{1}{3}\sum_{y\in\fth^*}
\sum_{x\in\fq}\zeta^{y{\rm Tr}(x^{\frac{q-1}{\wp^m}})}\nonumber\\
&=3^{\phi(\wp^m)-1}+\frac{1}{3}\sum_{y\in\fth^*}
\Big(\sum_{x\in\fq^*}\chi(yx^{\frac{q-1}{\wp^m}})+1\Big)\nonumber\\
&=3^{\phi(\wp^m)-1}+\frac{1}{3}\sum_{y\in\fth^*}
\left(S_{\wp^m}(y,0)+1\right)\nonumber\\
&=3^{\phi(\wp^m)-1}+\frac{1}{3}
\left(S_{\wp^m}(1,0)+S_{\wp^m}(2,0)+2\right).
\end{align}
Making use of Corollary \ref{cor2.5} and Equation (\ref{c2-2-2}), one has that
\begin{align}\label{c4-3}
S_{\wp^m}(1,0)+S_{\wp^m}(2,0)=\begin{cases}
\frac{q-1}{\wp^m}(2\wp^m-3\wp+3),&\text{if}\ \wp\equiv 1\pmod{3},\\
\frac{q-1}{\wp^m}(2\wp^m-3\wp),&\text{if}\ \wp\equiv -1\pmod{3}.
\end{cases}
\end{align}
It follows from (\ref{c4-2}) that
\begin{align}\label{c4-4-4}
n_0=\begin{cases}
\frac{q-1}{\wp^m}(\wp^m-\wp+1)+1,&\text{if}\ \wp\equiv 1\pmod{3},\\
\frac{q-1}{\wp^m}(\wp^m-\wp)+1,&\text{if}\ \wp\equiv -1\pmod{3}.
\end{cases}
\end{align}
Therefore, the length of $\mathcal{C}_D$ is obtained immediately.

Next, let us settle the weight distribution of $\mathcal{C}_D$. Let $D=\{d_1,d_2,\ldots,d_n\}$. One needs to understand the weight of codeword $\bm{c}_b=(\Tr(bd_1),\Tr(bd_2),\ldots,\Tr(bd_n))\in\mathcal{C}_D$ for any $b\in\fq^*$. Actually, the weight of $\bm{c}_b$ is equal to $n_0-N_b$, where $N_b$ is defined by
$$N_b:=\#\{x\in\fq:\ \Tr(x^{\frac{q-1}{\wp^m}})=0\ \text{and}\ \Tr(bx)=0\}.$$
We first compute $N_b$ for each $b\in\fq^*$, which will be done in what follows. For each $b\in\fq^*$, by definition, we have that
\begin{align*}
N_b&=\frac{1}{9}\sum_{x\in\fq}\sum_{y\in\fth}\zeta^{y{\rm Tr}(x^{\frac{q-1}{\wp^m}})}\sum_{z\in\fth}\zeta^{z{\rm Tr}(bx)}\nonumber\\
&=\frac{1}{9}\sum_{x\in\fq}\Big(\sum_{y\in\fth^*}\zeta^{y{\rm Tr}(x^{\frac{q-1}{\wp^m}})}+1\Big)\Big(\sum_{z\in\fth^*}\zeta^{z{\rm Tr}(bx)}+1\Big)\nonumber\\
&=\frac{1}{9}\Big(q+\sum_{z\in\fth^*}\sum_{x\in\fq}\chi(zbx)+
\sum_{y\in\fth^*}\sum_{x\in\fq}\chi(yx^{\frac{q-1}{\wp^m}})+
\sum_{y\in\fth^*}\sum_{z\in\fth^*}\sum_{x\in\fq}
\chi(yx^{\frac{q-1}{\wp^m}}+zbx)\Big).
\end{align*}
It implies that
\begin{align}\label{c4-4}
N_b=1/9&(S_{\wp^m}(1,0)+S_{\wp^m}(2,0)+S_{\wp^m}(1,b)
\nonumber\\
&+S_{\wp^m}(1,2b)+S_{\wp^m}(2,b)+S_{\wp^m}(2,2b)+q+6),
\end{align}
since $\sum_{x\in\fq}\chi(zbx)=0$ for any $z\in\fth^*$. Note that
$S_{\wp^m}(1,0)+S_{\wp^m}(2,0)$ has been evaluated in (\ref{c4-3}). One only needs to handle with $S_{\wp^m}(y,zb)$ for any $y,z\in\fth^*$. From Theorem \ref{thm3.3}, we know that
\begin{align}\label{c4-5}
S_{\wp^m}(y,zb)=\chi(c_{y,z})\sqrt{q}-\frac{\sqrt{q}+1}{\wp^m}S_{\wp^m}(c_{y,z}),
\end{align}
where $c_{y,z}:=y(zb)^{-\frac{q-1}{\wp^m}}$. Let $g$ be a primitive element of $\fq$. Since $b\in\fq^*$, let $b=g^{-i}$ for some integer $i$ with $0\le i\le q-2$. Then the values of $S_{\wp^m}(y,zb)$ can be obtained in the following parts.

$\bullet$ Part I: Evaluate $S_{\wp^m}(1,b)$. Recall that $\xi=g^{\frac{q-1}{\wp^m}}$. Then $c_{1,1}=b^{-\frac{q-1}{\wp^m}}=(g^{\frac{q-1}{\wp^m}})^i=\xi^i$.
In order to get the values of $\chi(c_{1,1})$ and $S(c_{1,1})$, one may consider the following cases.\\
{\sc Case 1}. $\wp^m\mid i$. In this case, $c_{1,1}=\xi^i=1$. So
$\chi(c_{1,1})=\chi(1)=\zeta^{{\rm Tr}(1)}=\zeta^{\phi(\wp^m)}$, i.e.,
\begin{align}\label{c4-6}
\chi(c_{1,1})=\begin{cases}
1,&\text{if}\ \wp\equiv 1\pmod{3},\\
\zeta^{(-1)^{m-1}},&\text{if}\ \wp\equiv -1\pmod{3}.
\end{cases}
\end{align}
Substituting (\ref{c4-6}) and (\ref{c2-11}) into (\ref{c4-5}) yields that
\begin{align*}
S_{\wp^m}(1,b)=\begin{cases}
-1+\frac{(\sqrt{q}+1)(p-1)(\zeta+2)}{\wp^m},&\text{if}\ \wp\equiv 1\pmod{3},\\
-\zeta+\frac{(\sqrt{q}+1)((\wp^m+\wp-2)\zeta-\wp^m+2\wp-1)}{\wp^m},&\text{if}\ \wp\equiv -1\pmod{3}\ \text{and}\ 2\nmid m,\\
-\zeta^2+\frac{(\sqrt{q}+1)((\wp^m+\wp-2)\zeta^2-\wp^m+2\wp-1)}{\wp^m},&\text{if}\ \wp\equiv -1\pmod{3}\ \text{and}\ 2\mid m.
\end{cases}
\end{align*}
{\sc Case 2}. $\wp^m\nmid i$. Write $i=s\wp^m+r$, where $s$ and $r$ are nonnegative integers with $1\le r\le \wp^m-1$. Clearly, $c_{1,1}=\xi^{r}$. First, let $1\le r\le (\wp-1)\wp^{m-1}$. Then there exists a unique integer pair $(k_0,i_0)$ with $1\le k_0\le \wp-1$ and $0\le i_0\le\wp^{m-1}-1$ such that $r=k_0\wp^{m-1}-i_0$. Let
$$c_{1,1}=\sum_{j=1}^{(\wp-1)\wp^{m-1}}a_j\xi^j$$
with each $a_j\in\fth$. Since $c_{1,1}=\xi^i=\xi^r$ one gets that
\begin{align}\label{c4-7}
c_{1,1}^{(i_0)}=(\underbrace{0,\ldots,0,1,0,\ldots,0}_{\wp-1}),\ \text{and}\ c_{1,1}^{(h)}=(\underbrace{0,0,\ldots,0}_{\wp-1})
\end{align}
for all $h$ with $h\ne i_0$ and $0\le h\le \wp^{m-1}-1$, where each $c_{1,1}^{(t)}$ is defined as (\ref{c2-2-10}), and the only one $1$ in the components of $c_{1,1}^{(i_0)}$ is the $k_0$-th component of this vector. Hence, by (\ref{c4-7}) and Corollary \ref{cor2.5}, we arrive at that
\begin{align}\label{c4-8}
S_{\wp^m}(c_{1,1})=\begin{cases}
(\wp-1)\zeta^2+\wp^m-\wp+1,&\text{if}\ \wp\equiv 1\pmod{3},\\
(\wp-2)\zeta^2+\wp^m-\wp-1,&\text{if}\ \wp\equiv -1\pmod{3}\ \text{and}\ 2\nmid m,\\
(\wp-2)\zeta+\wp^m-\wp-1,&\text{if}\ \wp\equiv -1\pmod{3}\ \text{and}\ 2\mid m.
\end{cases}
\end{align}
On the other hand, Lemma \ref{lem2.3} gives that
\begin{align}\label{c4-9}
\chi(c_{1,1})=\zeta^{{\rm Tr}(\xi^{r})}=\begin{cases}
\zeta^2,&\text{if}\ \wp^{m-1}\mid r\ \text{and}\ \wp\equiv 1\pmod{3},\\
\zeta^2,&\text{if}\ \wp^{m-1}\mid r, \wp\equiv -1\pmod{3}\ \text{and}\ 2\nmid m,\\
\zeta,&\text{if}\ \wp^{m-1}\mid r, \wp\equiv -1\pmod{3}\ \text{and}\ 2\mid m,
\\
1,&\text{if}\ \wp^{m-1}\nmid r.
\end{cases}
\end{align}
Putting Equations (\ref{c4-5}), (\ref{c4-8}) and (\ref{c4-9}) together, we conclude that
\begin{align}\label{c4-10}
S_{\wp^m}(1,b)=\begin{cases}
-\zeta^2+\frac{(\sqrt{q}+1)(-\wp^m+\wp-1)(\zeta+2)}{\wp^m},&\text{if}\ \wp\equiv 1\pmod{3},\\
-\zeta^2+\frac{(\sqrt{q}+1)((\wp^m-\wp+2)\zeta^2-\wp^m+\wp+1)}{\wp^m},&\text{if}\ \wp\equiv -1\pmod{3}\ \text{and}\ 2\nmid m,\\
-\zeta+\frac{(\sqrt{q}+1)((\wp^m-\wp+2)\zeta-\wp^m+\wp+1)}{\wp^m},&\text{if}\ \wp\equiv -1\pmod{3}\ \text{and}\ 2\mid m,
\end{cases}
\end{align}
when $\wp^{m-1}\mid r$, and
\begin{align}\label{c4-11}
S_{\wp^m}(1,b)=\begin{cases}
-1+\frac{(\sqrt{q}+1)(p-1)(\zeta+2)}{\wp^m},&\text{if}\ \wp\equiv 1\pmod{3},\\
-1+\frac{(\sqrt{q}+1)((\wp-2)\zeta+2\wp-1)}{\wp^m},&\text{if}\ \wp\equiv -1\pmod{3}\ \text{and}\ 2\nmid m,\\
-1+\frac{(\sqrt{q}+1)((\wp-2)\zeta^2+2\wp-1))}{\wp^m},&\text{if}\ \wp\equiv -1\pmod{3}\ \text{and}\ 2\mid m,
\end{cases}
\end{align}
when $\wp^{m-1}\nmid r$. In the following, let $(\wp-1)\wp^{m-1}<r\le \wp^m-1$. Write $r=\wp^{m}-1-t$ for some integer $t$ with $0\le t\le \wp^{m-1}-2$. Recall that the minimal polynomial of $\xi$ over $\fth$ is
$Q_{\wp^m}(x)=x^{(\wp-1)\wp^{m-1}}+x^{(\wp-2)\wp^{m-1}}+\cdots
+x^{\wp^{m-1}}+1$. So $$c_{1,1}=\xi^{r}=\xi^{\wp^m-1-t}=2\xi^{\wp^{m-1}-1-t}+2\xi^{2\wp^{m-1}-1-t}+\cdots
+2\xi^{(\wp-1)\wp^{m-1}-1-t}.$$
It then implies that
$$c_{1,1}^{(t+1)}=(\underbrace{2,2,\ldots,2}_{\wp-1}),\ \text{and}\  c_{1,1}^{(j)}=(\underbrace{0,0,\ldots,0}_{\wp-1})$$
for $j\ne t+1$. Employing Corollary \ref{cor2.5}, we then easily obtain $S_{\wp^m}(c_{1,1})$, in fact, which have the same values as (\ref{c4-8}). Note that $(\wp-1)\wp^{m-1}<r\le \wp^m-1$. It follows from Lemma \ref{lem2.3} that $\chi(c_{1,1})=1$. Therefore, by (\ref{c4-5}) we derive $S{\wp^m}(1,b)$, which is consistent with (\ref{c4-11}).

Thus, the values of $S_{\wp^m}(1,b)$ have been completely determined.

$\bullet$ Part II: Evaluate $S_{\wp^m}(1,2b)$. Note that $c_{1,2}=(2b)^{-\frac{q-1}{\wp^m}}=b^{-\frac{q-1}{\wp^m}}=c_{1,1}$ since $\frac{q-1}{\wp^m}$ is an even number. It follows that $S_{\wp^m}(1,2b)=S_{\wp^m}(1,b)$. So $S_{\wp^m}(1,2b)$ is also obtained.

$\bullet$ Part III: Evaluate $S_{\wp^m}(2,b)$ and $S_{\wp^m}(2,2b)$. First, similar to Part II, we have $S_{\wp^m}(2,b)=S_{\wp^m}(2,2b)$. Note that $c_{2,1}=2b^{-\frac{q-1}{\wp^m}}=2\xi^i$. Hence, if let $S_{\wp^m}(1,b)$ be a function with variable $\zeta$, denoted by $S_{\wp^m}(1,b;\zeta)$, then from Part I, we readily find that $S_{\wp^m}(2,b)=S_{\wp^m}(1,b;\zeta^2)$. Finally, $S_{\wp^m}(2,b)$ and $S_{\wp^m}(2,2b)$ are settled.

To sum up, substituting the results of Parts I, II and III, and (\ref{c4-5}) into (\ref{c4-4}) yields that
\begin{align*}
N_b=\begin{cases}
\frac{\wp^{m}-\wp+1}{3\wp^m}q+\frac{2(\wp-1)}{3\wp^m}\sqrt{q}
+\frac{\wp-1}{\wp^m},&\text{if}\ \wp\equiv 1\pmod{3},\\
\frac{\wp^{m-1}-1}{3\wp^{m-1}}q-\frac{2(\wp^{m-1}-1)}{3\wp^{m-1}}\sqrt{q}
+\frac{1}{\wp^{m-1}},&\text{if}\ \wp\equiv -1\pmod{3},
\end{cases}
\end{align*}
if $\wp^m\mid {\rm Ind}_g(b^{-1})$,
\begin{align*}
N_b=\begin{cases}
\frac{\wp^{m}-\wp+1}{3\wp^m}q+\frac{2(-\wp^2+\wp-1)}{3\wp^m}\sqrt{q}
+\frac{\wp-1}{\wp^m},&\text{if}\ \wp\equiv 1\pmod{3},\\
\frac{\wp^{m-1}-1}{3\wp^{m-1}}q-\frac{2(\wp^{m-1}-1)}{3\wp^{m-1}}\sqrt{q}
+\frac{1}{\wp^{m-1}},&\text{if}\ \wp\equiv -1\pmod{3},
\end{cases}
\end{align*}
if $\wp^m\nmid {\rm Ind}_g(b^{-1})$ and $\wp^{m-1}\mid \langle {\rm Ind}_g(b^{-1})\rangle_{\wp^m}$, and
\begin{align*}
N_b=\begin{cases}
\frac{\wp^{m}-\wp+1}{3\wp^m}q+\frac{2(\wp-1)}{3\wp^m}\sqrt{q}
+\frac{\wp-1}{\wp^m},&\text{if}\ \wp\equiv 1\pmod{3},\\
\frac{\wp^{m-1}-1}{3\wp^{m-1}}q+\frac{2}{3\wp^{m-1}}\sqrt{q}
+\frac{1}{\wp^{m-1}},&\text{if}\ \wp\equiv -1\pmod{3},
\end{cases}
\end{align*}
if $\wp^m\nmid {\rm Ind}_g(b^{-1})$ and $\wp^{m-1}\nmid \langle {\rm Ind}_g(b^{-1})\rangle_{\wp^m}$, where ${\rm Ind}_g(b^{-1})$ denotes the index of $b^{-1}$ with respect to $g$ and $\langle {\rm Ind}_g(b^{-1})\rangle_{\wp^m}$ denotes the least nonnegative residue of ${\rm Ind}_g(b^{-1})$ modulo $\wp^m$. It follows that $\mathcal{C}$ is a two-weight code either having weights
$$w_1=\frac{2(\wp^m-\wp+1)}{3\wp^m}q-\frac{2(\wp-1)}{3\wp^m}\sqrt{q},\ w_2=\frac{2(\wp^m-\wp+1)}{3\wp^m}q+\frac{2(\wp^2-\wp+1)}{3\wp^m}\sqrt{q}$$
if $\wp\equiv 1\pmod{3}$, or having weights
$$w_1=\frac{2(\wp^{m-1}-1)}{3\wp^{m-1}}q+\frac{2(\wp^{m-1}-1)}{3\wp^{m-1}}\sqrt{q},\ w_2=\frac{2(\wp^{m-1}-1)}{3\wp^{m-1}}q-\frac{2}{3\wp^{m-1}}\sqrt{q}$$
if $\wp\equiv -1\pmod{3}$.

Let $A_{w_i}$ be the number of codewords with weight $w_i$, $i=1,2$. We now are able to determine $A_{w_i}$. First, from the above computation we know that the weight of $\bm{c}_b$ is always greater than zero for any $b\in\fq^*$. It implies that the map
$$x\mapsto (\Tr(xd_1),\Tr(xd_2),\ldots,\Tr(xd_n))$$
from $\fq$ to $\mathcal{C}_D$ is a bijection. It gives us $\#\mathcal{C}_D=q$. Consequently, the dimension of $\mathcal{C}_D$ is $(\wp-1)\wp^{m-1}$. On the other hand, it is observed that there is no codeword of weight one in $\mathcal{C}_D^{\perp}$ since the map $\Tr$ is onto. Therefore, by the first two Pless Power Moments, we have
$$\begin{cases}
A_{w_1}+A_{w_2}=q-1,\\
A_{w_1}w_1+A_{w_2}w_2=\frac{1}{3}(2n_0-2)q,
\end{cases}$$
where $n_0$ is given by (\ref{c4-4-4}).
It follows that
\begin{align*}
\begin{cases}
A_{w_1}=\frac{(\wp^m-\wp+1)(q-1)}{\wp^m},\ A_{w_2}=\frac{(\wp-1)(q-1)}{\wp^m},&\text{if}\ \wp\equiv 1\pmod{3},\\
A_{w_1}=\frac{q-1}{\wp^{m-1}},\ A_{w_2}=\frac{(\wp^{m-1}-1)(q-1)}{\wp^{m-1}},&\text{if}\ \wp\equiv -1\pmod{3},
\end{cases}
\end{align*}
as desired. It completes the proof of Theorem \ref{thm4.1}.
\end{proof}
\begin{rem}
Theorem \ref{thm3.3} gives the explicit evaluation of $S_N(a,b)$ if the representation coefficients of $ab^{-\frac{q-1}{N}}$ by the basis
$(\xi,\xi^2,\ldots,\xi^{\phi(N)})$ are computable. However, it seems not easy to get the representation coefficients of $ab^{-\frac{q-1}{N}}$ for any $a,b\in\mathbb{F}_q^*$. The proof of Theorem \ref{thm4.1} provides a method of computing the representation coefficients of $ab^{-\frac{q-1}{N}}$ for $a\in\fp^*$ and $b\in\fq^*$.
\end{rem}

\section{The dual codes of $\mathcal{C}_D$}
Let $m$ be an integer greater than $1$, and $\wp$ be an odd prime such that $3$ is a primitive root modulo $\wp^m$. Let $\mathcal{C}_D$ be the code defined in Section 4, and let $\mathcal{C}_D^{\perp}$ be the dual code of $\mathcal{C}_D$. Then we have the following results.
\begin{thm}\label{thm5.1}
The dual code $\mathcal{C}_D^{\perp}$ of the two-weight code $\mathcal{C}_D$ is a ternary code with parameters $\left[n,n-\wp^{m-1}(\wp-1),2\right]$, where
\begin{align*}
n=\begin{cases}
\frac{q-1}{\wp^m}(\wp^m-\wp+1),&\text{if}\ \wp\equiv 1\pmod{3},\\
\frac{q-1}{\wp^m}(\wp^m-\wp),&\text{if}\ \wp\equiv -1\pmod{3}.
\end{cases}
\end{align*}
\end{thm}
\begin{proof}
Note that $\dim(\mathcal{C}_D)+\dim(\mathcal{C}_D^{\perp})=n$. So the dimension of $\mathcal{C}_D^{\perp}$ follows from Theorem \ref{thm4.1}. In the following, let us determine the minimum distance $d$ of $\mathcal{C}_D^{\perp}$. It is observed that $d\ge 2$ since the map $\Tr$ is onto. Let $D=\{d_1,d_2,\ldots,d_n\}$. We know that $\#D>2$ since $m\ge 2$. Taking an element $d_i$ of $D$, one finds that $-d_i\in D$ and $d_i\ne -d_i$. Then $-d_i=d_j$ for some $j\ne i$. It is easily checked that $\bm{c}=(c_1,c_2,\ldots,c_n)\in\mathcal{C}_D^{\perp}$, where $c_i=c_j=1$ and $c_k=0$ for all $k\ne i,j$. It then follows that $d=2$. The proof of Theorem \ref{thm5.1} is done.
\end{proof}
\begin{defn}\label{defn5.5}
A code $\mathcal{C}$ with parameter $[n,k,d]$ is called \textit{optimal} if no $[n,k,d+1]$ code exists.
\end{defn}
Then by the sphere packing bound, the last result of the paper follows immediately.
\begin{cor}\label{cor5.3}
The dual code $\mathcal{C}_D^{\perp}$ is optimal.
\end{cor}
\begin{proof}
Let $n$ and $d$ be the length and minimum distance of $\mathcal{C}_{D}^{\perp}$. The sphere packing bound tells that 
$$\big|\mathcal{C}_{\alpha,\beta}^{\perp}\big|\le \frac{3^n}{\sum_{i=0}^t\binom{n}{i}2^i},$$
where $t$ denotes the greatest integer no more than $\frac{d-1}{2}$. Note that $d=2$. It follows that the dual code $\mathcal{C}_{\alpha,\beta}^{\perp}$ is optimal if $n>\frac{q-1}{2}$. So the desired result follows from Theorem \ref{thm5.1}.
\end{proof}

\section{Conclusion}
In this paper, we provided explicit evaluations of binomial Weil sums $S_N(a,b)$ for all various values of $N$, assuming that $p$ is a primitive root modulo 
$N$. Our results significantly extend previous work by removing additional conditions and offering a more general and elementary approach. Moreover, we constructed a new class of ternary linear codes, for which the weight distribution was fully determined. This weight distribution analysis demonstrated the optimality of the dual codes with respect to the sphere packing bound. 

In addition, the method presented in this paper is expected to enable the construction of general $p$-ary codes from the evaluations of the exponential sums, and to facilitate the determination of their Hamming weight distributions. Specifically, let $N \in \{4, \wp^m, 2\wp^m\}$, and let $p$ be a general odd prime such that $p$ is a primitive root modulo $N$. By replacing the defining set 
\[
\{x \in \mathbb{F}_q^* : \text{Tr}(x^{\frac{q-1}{\wp^m}}) = 0\}
\]
with the more general set
\[
D = \{x \in \mathbb{F}_q^* : \text{Tr}\left( x^{\frac{q-1}{N}} + \beta x \right) = \alpha\},
\]
where $\alpha \in \mathbb{F}_p$ and $\beta \in \mathbb{F}_q$, similar results for the code $\mathcal{C}_{D}$ can be derived using analogous techniques. Recently, leveraging the results of this paper, the first author \cite{[C]} established the weight distribution of $\mathcal{C}_{D}$ for any $\alpha \in \mathbb{F}_p$ and $\beta \in \mathbb{F}_q$ when $p = 3$ and $N = 2\wp^m$.

\section*{acknowledgment}
The first author would like to express his sincere gratitude to Prof. Arne Winterhof at RICAM, Austrian Academy of Sciences, for his invaluable guidance and support during his academic visit.

%\section*{Acknowledgment}
%The author would like to thank the anonymous referees for their helpful comments which improved the presentation of the paper.

\end{document}